\documentclass[a4,11pt]{article}
\usepackage{graphicx}
\usepackage{amssymb}
\usepackage{url}
\usepackage{latexsym}
\usepackage{amssymb}
\usepackage{here}
\usepackage{wrapfig}
\usepackage{fancybox}
\usepackage{ascmac}
\usepackage{amsmath}
\usepackage{multicol}
\usepackage{mathtools}
\usepackage{amsthm}
\usepackage{color}
\usepackage{sidenotes}
\usepackage{cite}

\newcommand{\mc}{\mathcal}

\newcommand{\Honezero}{H^1_0(\Omega)}

\newcommand{\Hinv}{H^{-1}}

\newtheorem{theorem}{Theorem}[section]
\newtheorem{lemma}[theorem]{Lemma}
\newtheorem{remark}[theorem]{Remark}
\newtheorem{corollary}[theorem]{Corollary}

\allowdisplaybreaks[4]

\begin{document}
\begin{center}
\noindent{\bf\Large A new formulation for the numerical proof of the existence of solutions to elliptic problems
}\vspace{10pt}\\
 {\normalsize Kouta Sekine$^{1,*}$, Mitsuhiro T. Nakao$^{2}$, Shin'ichi Oishi$^{3}$}\vspace{5pt}\\
 {\it\normalsize $^{1}$Faculty of Information Networking for Innovation and Design, Toyo University, 1-7-11 Akabanedai, kita-ku, Tokyo 115-0053, Japan.\\
 $^{2}$Faculty of Science and Engineering, Waseda University, 3-4-1 Okubo, Tokyo 169-8555, Japan\\
 $^{3}$Department of Applied Mathematics, Faculty of Science and Engineering, Waseda University, 3-4-1 Okubo, Tokyo 169-8555, Japan}
\end{center}
{\bf Abstract}. 
Infinite-dimensional Newton methods can be effectively used to derive numerical proofs of the existence of solutions to partial differential equations (PDEs).
In computer-assisted proofs of PDEs, the original problem is transformed into the infinite Newton-type fixed point equation $w = - {\mathcal L}^{-1} {\mathcal F}(\hat{u}) + {\mathcal L}^{-1} {\mathcal G}(w)$, where ${\mathcal L}$ is a linearized operator, ${\mathcal F}(\hat{u})$ is a residual, and ${\mathcal G}(w)$ is a local Lipschitz term.
Therefore, the estimations of $\| {\mathcal L}^{-1} {\mathcal F}(\hat{u}) \|$ and $\| {\mathcal L}^{-1}{\mathcal G}(w) \|$ play major roles in the verification procedures.  

In this paper, using  a similar concept as the `Schur complement' for matrix problems, we represent the inverse operator ${\mathcal L}^{-1}$ as an infinite-dimensional operator matrix  that can be decomposed into two parts, one finite dimensional and one infinite dimensional.
This operator matrix yields a new effective realization of the infinite-dimensional Newton method, enabling a more efficient verification procedure compared with existing methods for the solution of elliptic PDEs.
We present some numerical examples that confirm the usefulness of the proposed method.
Related results obtained from the representation of the operator matrix as ${\mathcal L}^{-1}$ are presented in the appendix.
\renewcommand{\thefootnote}{\fnsymbol{footnote}}
\footnote[0]{{\it E-mail address:} $^{*}$\texttt{ k.sekine@computation.jp}\\[-3pt]}
\renewcommand\thefootnote{*\arabic{footnote}}

\section{Introduction}
In this paper, we study a new approach to proving the existence of solutions for elliptic problems. The proposed approach offers an improvement over existing numerical verification methods. 
We consider computer-assisted existence proofs for the nonlinear elliptic boundary value problem
\begin{align}
	\label{main_problem}
	\left\{
		\begin{array}{ll}
		-\Delta u = f( u ) & {\rm in~~} \Omega,\\
		u=0 & {\rm on~~} \partial\Omega,
		\end{array}
	\right.
\end{align}
where $\Omega \subset {\mathbb R}^n (n = 1,2,3)$ is a bounded domain with a Lipschitz boundary, and $f:H^1_0(\Omega) \rightarrow H^{-1}(\Omega)$ is a given nonlinear function.
Equation (\ref{main_problem}) is a basic case of a semi-linear elliptic partial differential equation (PDE), for which many computer-assisted proof methods have been developed \cite{nakao1988numerical, nakao1990numerical, plum1991bounds, plum1994enclosures, oishi1995numerical, nakao2001numerical, nakao2004efficient, nakao2005numerical, plum2008, plum2009computer, nakao2011numerical}.
These methods are intended to prove the existence of solutions based on the fixed point theorem in  the Sobolev spaces.
Throughout this paper, denoting the first-order $ L^2 $ Sobolev space by $ H^1(\Omega) $, we define $ H^1_0(\Omega) := \{ u \in H^1(\Omega) : u = 0~\mbox{on} ~ \partial \Omega\} $ with the inner product $(u, w)_{H^1_0}:=(\nabla u, \nabla w)_{L^2}$, 
and $ \Hinv $ denotes the topological dual of $ \Honezero$.
We now categorize the verification methods developed so far to clarify the significance and advantages of the method described in this paper:
\begin{itemize}
  \item FN method: Applying the Newton method only for the finite-dimensional part (e.g., FN-Int \cite{nakao1990numerical, nakao2001numerical, nakao2011numerical}, FN-Norm \cite{nakao2004efficient, nakao2011numerical}).
  \item IN method: Using the infinite-dimensional Newton's method (e.g., Newton--Kantorovich-like theorem \cite{plum1994enclosures, plum2008, plum2009computer}, IN-Linz \cite{nakao2005numerical, nakao2011numerical, oishi1995numerical}, Newton--Kantorovich theorem \cite{takayasu2013verified}).
\end{itemize}

In the FN method, with an appropriate setting of the finite-dimensional subspace $V_h \subset H^1_0(\Omega)$, we first consider the Ritz projection $R_h : H^1_0(\Omega) \rightarrow V_h$ defined by 
\begin{align*}
	((I - R_h) u, v_h)_{H^1_0} = 0 \quad \forall v_h \in V_h
\end{align*}
for $u \in H^1_0(\Omega)$. Using this projection, the problem is decomposed into two parts, one finite dimensional and the other infinite dimensional. Let $\psi_1, \! \cdots, \! \psi_N $ be a basis of $V_h$, and let
$V_\bot := \{ u \in H^1_0(\Omega) ~ | ~ (u, v_h)_{H^1_0(\Omega)} = 0, ~ v_h \in V_h  \}$ be an orthogonal complement of $V_h$. For a given approximate solution $\hat{u} \in V_h$, setting $w := u - \hat{u}$, $w_h := R_h w$, and $w_\bot := (I - R_h)w$, the FN method uses the following fixed point formulation:\\
find $w_h \in V_h, ~ w_\bot \in V_\bot$ satisfying
\begin{align}
	\label{eq:FN_Method}
	\left\{
		\begin{array}{l}
				w_h = R_h {\mathcal A}^{-1} (f (w_h + w_\bot + \hat{u}) - {\mathcal A} \hat{u}), \\
				w_\bot = (I - R_h) {\mathcal A}^{-1} (f (w_h + w_\bot + \hat{u}) - {\mathcal A} \hat{u}),
		\end{array}
	\right.
\end{align}
where ${\mc A}: H^1_0(\Omega) \rightarrow H^{-1}(\Omega)$ denotes the weak Laplace operator.  In \cite{nakao1988numerical}, the candidate set of solutions is set to 
\begin{align}
	\label{CandidateSet_wh}
		W_h &:= \left\{ \sum_{i = 1}^{N} W_i \psi_i \subset V_h ~ | ~ W_i ~ \mbox{is a closed interval in } {\mathbb R} ~   \right\}, \\
	\label{CandidateSet_wbot}
		W_\bot &:= \{ w_{\bot} \in V_{\bot} ~ | ~ \| w_{\bot} \|_{H^1_0}, \le \alpha \},
\end{align}
 and the fixed point theorem is applied to (\ref{eq:FN_Method}) in \cite{nakao1988numerical} to verify a solution in the set $W= W_h + W_\bot$.  To confirm the verification condition in the fixed point theorem, the verified computation of solutions for a linear system of equations and the constructive a priori error estimates for the Ritz projection play an essential role. 
This was the first approach to numerical  verification, but is based on a sequential iteration method that differs from the FN method.

The FN method applies Newton's method to the finite-dimensional part of (\ref{eq:FN_Method}) as follows:\\
Let $f'[\hat{u}]$ be the Fr\'echet derivative at $\hat{u}$ of the nonlinear term $f(u)$, and let ${\mathcal L}: H^1_0(\Omega) \rightarrow H^{-1}(\Omega)$ be a linear operator defined as
\begin{align}
	\label{def:linearize_operator}
	{\mathcal L} := {\mathcal A} - f'[\hat{u}].
\end{align}
Furthermore, the finite-dimensional operator $T : V_h \rightarrow V_h$ is defined as 
\begin{align}
	\label{def:T}
	T := R_h {\mathcal A}^{-1} {\mathcal L}|_{V_h},
\end{align}
and $T$ is assumed to be nonsingular.
Then, we can rewrite the finite-dimensional part of (\ref{eq:FN_Method}) as 
\begin{align}
	\label{eq:FN_Method_finite}
		w_h = T^{-1} R_h {\mathcal A}^{-1} (f (w_h + w_\bot + \hat{u}) - {\mathcal A} \hat{u} - f'[\hat{u}]w_h).
\end{align}

This FN method was developed in \cite{nakao1990numerical}, and it has been confirmed that Newton's method is actually more effective for the verification of $w_h$ than the method in \cite{nakao1988numerical}.
Additionally, by using the verified matrix computations in (\ref{eq:FN_Method_finite}), the calculation cost is not unreasonable. 
However, Newton's method cannot be expected to be effective for the infinite dimensional $w_{\bot}$. This causes difficulties in the verification process. For example, if first derivatives such as $(b \cdot \nabla) u $ are included in $f(u)$, the verification becomes inefficient and often fails. Namely, as Newton's method is ineffective on the primary term of the infinite-dimensional part $w_\bot$, there is an explosive expansion of the iteration and verification fails. In \cite{nakao2004efficient}, the FN-Norm method was proposed. This ensures more effective verification by setting the candidate set $ W_h $ of the finite-dimensional part as 
\begin{align*}
	W_h &:= \left\{ w_h \in V_h ~ | ~ \| w_h \|_{H^1_0} \le \gamma ~   \right\}.
\end{align*}
However, the problem is not essentially resolved,  because the first-order term of the infinite-dimensional part $w_\bot$ remains, and this reduces the effect of Newton's method.

In contrast, the IN method assumes that the linearized operator ${\mathcal L}$ is nonsingular and considers the following fixed point equation:\\
 find $w \in H^1_0(\Omega)$ such that
\begin{align}
	\label{eq:IN_Method}
		w = - {\mathcal L}^{-1} {\mathcal F}(\hat{u}) + {\mathcal L}^{-1} {\mathcal G}(w),
\end{align}
where 
\begin{align}
	\label{def:Residual}
{\mathcal F}(\hat{u}) := {\mathcal A}\hat{u} - f(\hat{u})
\end{align}
and
\begin{align}
	\label{def:Lipsitz}
{\mathcal G}(w) :=  f'[\hat{u}] w + f(\hat{u}) - f (w + \hat{u}).
\end{align}
As (\ref{eq:IN_Method}) is a Newton-type formulation in the infinite-dimensional sense, the linear term with respect to $w$, which is a shortcoming of the FN method, is no longer present in the right-hand side of the equation.
Therefore, a Newton--Kantorovich-like theorem can be derived using the fixed point equation (\ref{eq:IN_Method}).
However, as the inverse operator ${\mathcal L}^{-1}$ cannot be calculated directly, it is necessary, as a decomposition from the verification for nonlinear problems, to show the invertibility of  ${\mathcal L}$ and estimate the operator norm $\| {\mathcal L}^{-1} \|$.
Thus, the evaluation of $|| {\mathcal L}^{- 1} ||$ is the major task in the verification procedures of the IN method,  and has been studied by many researchers since 1991 (e.g., \cite{plum1991bounds, plum1994enclosures, oishi1995numerical, nakao2005numerical, watanabe2013posteriori, tanaka2014verified, nakao2015some, watanabe2016norm,kinoshita2019alternative,watanabe2019improved}).
In general, the IN method defines the candidate set as 
\begin{align}
	\label{CandidateSet_IN}
		W &:= \{ w \in H^1_0(\Omega) ~ | ~ \| w \|_{H^1_0} \le \rho \}.
\end{align}
Therefore, compared with the FN method, the IN method is overestimated for finite-dimensional parts that can be computed directly.

We now describe a new approach that incorporates the advantages of both  the FN  and IN methods. 
In this paper, we basically consider the IN method based on \eqref{eq:IN_Method}, but we propose a verification method {\it without} estimating the norm $\| {\mathcal L}^{-1} \|$.
Through a computational procedure that avoids the direct evaluation of the operator norm $\| {\mathcal L}^{-1} \|$, a highly accurate and efficient verification can be expected.
We decompose  (\ref{eq:IN_Method}) into finite-dimensional and infinite-dimensional parts: namely, $(w_h, w_\bot)^T \in V_h \times V_\bot$ such that
\begin{align}
	\label{eq:main_fixed_point_form}
	\left( 
	\begin{array}{c}
				w_h \\
				w_\bot
	\end{array} 
	\right) = - \left(
	\begin{array}{c}
				R_h {\mathcal L}^{-1} {\mathcal F}(\hat{u})  \\
				(I - R_h) {\mathcal L}^{-1} {\mathcal F}(\hat{u})
	\end{array} \right) + \left(
	\begin{array}{c}
				R_h {\mathcal L}^{-1} {\mathcal G}(w_h + w_\bot) \\
				(I - R_h) {\mathcal L}^{-1} {\mathcal G}(w_h + w_\bot)
	\end{array}
	\right)  . 
\end{align}
However, we cannot directly calculate ${\mathcal L}^{-1}$ for the same reason as in the existing IN method.
To overcome this difficulty, we define the operator matrix
\begin{align}
	\label{def:H1}
		H=\left(\begin{array}{c c}
			H_{11} & H_{12} \\
			H_{21} & H_{22}
		\end{array}
		\right)
		: V_h \times V_\bot \rightarrow V_h \times V_\bot
\end{align}
satisfying 
\begin{align}
	\label{def:H2}
		\left(
		\begin{array}{c}
					R_h {\mathcal L}^{-1} g  \\
					(I - R_h) {\mathcal L}^{-1} g
		\end{array} \right) = \left( \begin{array}{c c}
					H_{11} & H_{12}  \\
					H_{21} & H_{22}
		\end{array}
		\right) \left( \begin{array}{c}
					R_h {\mathcal A}^{-1} g  \\
					(I - R_h) {\mathcal A}^{-1} g
		\end{array} \right), ~ \forall g \in H^{-1}(\Omega).
\end{align}
Here, we apply the fixed point theorem to the following equation:
\begin{align}
	\label{eq:fixed_point_form_with_H}
	\left( \begin{array}{c}
				w_h \\
				w_\bot
	\end{array} 
	\right) = -H \left(
	\begin{array}{c}
				R_h {\mathcal A}^{-1} {\mathcal F}(\hat{u})  \\
				(I - R_h) {\mathcal A}^{-1} {\mathcal F}(\hat{u})
	\end{array} \right) + H \left(
	\begin{array}{c}
				R_h {\mathcal A}^{-1} {\mathcal G}(w_h + w_\bot) \\
				(I - R_h) {\mathcal A}^{-1} {\mathcal G}(w_h + w_\bot)
	\end{array} \right)  
\end{align}
with candidate sets (\ref{CandidateSet_wh}) and (\ref{CandidateSet_wbot}).
Note that the right-hand side of this fixed point equation no longer has linear terms in  $w_h$ and $w_\bot$, and that we can directly calculate the finite-dimensional part. Therefore, the proposed method removes the disadvantages of the FN and IN methods.
For the actual implementation of the verification procedure, it is necessary to obtain a more detailed form of the operator matrix $H$. Thus, we consider a concrete construction of this matrix below.

The remainder of this paper is organized as follows. 
Section 2 describes how to construct the operator matrix $H$.
In Section 3, we present the results of numerical experiments using the operator matrix $H$.
The appendix describes why the proposed method offers an improvement over previous techniques based on some useful results.

\section{Constitution of the inverse operator matrix $H$}
This section presents a detailed description of the actual construction  of the operator matrix $H$ defined by (\ref{def:H1}) and (\ref{def:H2}).
The basic idea comes from the concept of the `Schur complement' for matrices.

We consider a solution $\phi$ that satisfies the linear equation 
\begin{align}
	\label{eq:linear_prog}
		{\mathcal L} \phi = g
\end{align}
for a given $g \in H^{-1}(\Omega)$.
We denote 
\begin{align*}
	\phi_h &:= R_h \phi,& \phi_\bot &:= ( I - R_h) \phi, \\
	{\mathcal A}_h^{-1} &:= R_h {\mathcal A}^{-1}, & {\mathcal A}_{\bot}^{-1} &:= (I - R_h) {\mathcal A}^{-1}.
\end{align*}
We multiply ${\mathcal A}^{-1}$ from the left on both sides of (\ref{eq:linear_prog}), and decompose the result into the finite- and infinite-dimensional parts using the Ritz projection $R_h$ as follows:
\begin{align*}
	& \left\{
		\begin{array}{l}
			R_h {\mathcal A}^{-1}{\mathcal L} (\phi_h + \phi_\bot) = {\mathcal A}_h^{-1}g \\
			(I - R_h ) {\mathcal A}^{-1}{\mathcal L} (\phi_h + \phi_\bot) = {\mathcal A}_{\bot}^{-1}g
		\end{array}
	\right. \\
	\Leftrightarrow & \left\{
		\begin{array}{l}
			T \phi_h - R_h {\mathcal A}^{-1}f'[\hat{u}] \phi_\bot = {\mathcal A}_h^{-1}g \\
			- (I - R_h ) {\mathcal A}^{-1}f'[\hat{u}] \phi_h + (I_{V_\bot} - (I - R_h){\mathcal A}^{-1}f'[\hat{u}] ) \phi_\bot = {\mathcal A}_{\bot}^{-1}g
		\end{array}
	\right. ,
\end{align*}
where $I_{V_{\bot}}$ is an identity operator on $V_{\bot}$.
Furthermore, transforming the above equation using the operator matrix yields
\begin{align}
	\label{eq:linear_prog_block}
	\left(
	\begin{array}{c c}
		T & -{\mathcal A}_h^{-1} f'[\hat{u}] |_{V_\bot} \\
		-{\mathcal A}_{\bot}^{-1} f'[\hat{u}] |_{V_h} & I_{V_{\bot}} - {\mathcal A}_{\bot}^{-1} f'[\hat{u}] |_{V_\bot}
	\end{array}
	\right)
	\left(
	\begin{array}{c}
		\phi_h \\
		\phi_{\bot}
	\end{array}
	\right) = 
	\left(
	\begin{array}{c}
		{\mathcal A}_h^{-1} g \\
		{\mathcal A}_{\bot}^{-1} g
	\end{array}
	\right).
\end{align}
Additionally, we define the $2 \times 2$ block operator matrix $D$ by 
\begin{align}
	\label{def:operator_matrix_D}
	D := \left(
	\begin{array}{c c}
		T & -{\mathcal A}_h^{-1} f'[\hat{u}] |_{V_\bot} \\
		-{\mathcal A}_{\bot}^{-1} f'[\hat{u}] |_{V_h} & I_{V_{\bot}} - {\mathcal A}_{\bot}^{-1} f'[\hat{u}] |_{V_\bot}
	\end{array}
	\right).
\end{align}
Moreover, if the operators ${\mathcal L}$ and $D$ are nonsingular, then we have
\begin{align*}
	\left(
	\begin{array}{c}
		\phi_h \\
		\phi_{\bot}
	\end{array}
	\right) = \left(
	\begin{array}{c}
		R_h {\mathcal L}^{-1} g \\
		(I - R_h){\mathcal L}^{-1} g
	\end{array}
	\right) = D^{-1} \left(
	\begin{array}{c}
		{\mathcal A}_h^{-1} g \\
		{\mathcal A}_{\bot}^{-1} g
	\end{array}
	\right)
\end{align*}
and $H$ is equal to $D^{-1}$ from (\ref{def:H2}).

We first present a sufficient condition for the nonsingularity of operators $D$ and ${\mathcal L}$, which also gives a detailed expression of $H (= D^{-1})$.

\begin{theorem}\label{thm:inverse_theorem}
The finite-dimensional operator $T : V_h \rightarrow V_h$ defined as (\ref{def:T}) is assumed to be nonsingular.
Let $S: V_{\bot} \rightarrow V_{\bot}$ be a linear operator defined as
\begin{align}
	\label{def:S}
		S := I_{V_\bot} - {\mathcal A}_{\bot}^{-1} f'[\hat{u}] |_{V_\bot} - {\mathcal A}_{\bot}^{-1} f'[\hat{u}] |_{V_h} T^{-1} {\mathcal A}_h^{-1} f'[\hat{u}] |_{V_\bot}.
\end{align}
If $S$ is nonsingular, then the operators $D$ and  ${\mathcal L}$, defined by (\ref{def:operator_matrix_D}) and (\ref{def:linearize_operator}), respectively, are also nonsingular and the solution $(\phi_h, \phi_{\bot})^T \in V_h \times V_{\bot}$ that satisfies (\ref{eq:linear_prog_block}) is represented as 
{\small
\begin{align*}
	& \left(
	\begin{array}{c}
		\phi_h \\
		\phi_{\bot}
	\end{array}
	\right) \\ 
	= & \left(
	\begin{array}{c c}
		T^{-1} + T^{-1}{\mathcal A}^{-1}_h f'[\hat{u}]|_{V_{\bot}} S^{-1} {\mathcal A}^{-1}_{\bot} f'[\hat{u}]|_{V_h}T^{-1} ~~ & ~ T^{-1}{\mathcal A}^{-1}_h f'[\hat{u}]|_{V_{\bot}}S^{-1} \\
		S^{-1}  {\mathcal A}^{-1}_{\bot} f'[\hat{u}]|_{V_h}T^{-1} & S^{-1}
	\end{array}
	\right)
	\left(
	\begin{array}{c}
		{\mathcal A}_h^{-1} g \\
		{\mathcal A}_{\bot}^{-1} g
	\end{array}
	\right).
\end{align*}
}
\end{theorem}

\begin{proof}. 
We first show that the operator matrix $D$ is nonsingular.
Let $Y: V_{\bot} \rightarrow V_h$, $Z: V_h \rightarrow V_{\bot}$, and $G: V_{\bot} \rightarrow V_{\bot}$ be bounded linear operators defined as
\begin{align}
\label{def:YZG}
Y := - {\mathcal A}^{-1}_h f'[\hat{u}] |_{V_{\bot}}, ~ Z := - {\mathcal A}^{-1}_{\bot} f'[\hat{u}] |_{V_h} \mbox{and}\;G := I - {\mathcal A}^{-1}_{\bot} f'[\hat{u}] |_{V_{\bot}},
\end{align}
respectively.
We rewrite the operator matrix $D$ and the bounded operator $S$ as follows: 
\begin{align*}
D = \left( \begin{array}{cc}
T & Y \\
Z & G
\end{array} \right) ~ \mbox{and} ~ S = G - ZT^{-1}Y.
\end{align*}
We note that the operator $S$ is corresponding to the Schur complement of $T$ of the matrix operator $D$.

As $T$ and $S$ are nonsingular, we can define the operator matrix $\bar{D}$ as
\begin{align*}
\bar{D} := \left( \begin{array}{cc}
T^{-1} + T^{-1} Y S^{-1} Z T^{-1}  ~~ & ~ -T^{-1} Y S^{-1}  \\
-S^{-1} Z T^{-1} & S^{-1}
\end{array} \right).
\end{align*}
If $\bar{D} D$ and $D \bar{D}$ become identity operator matrices, then this implies that  $D$ is nonsingular and its inverse matrix coincides with $\bar{D}$. Observe that
\begin{eqnarray*}
\bar{D}D &=& \left( \begin{array}{cc}
T^{-1} + T^{-1} Y S^{-1} Z T^{-1}   ~&~ -T^{-1} Y S^{-1}  \\
-S^{-1} Z T^{-1} & S^{-1}
\end{array} \right) \left( \begin{array}{cc}
T & Y \\
Z & G
\end{array} \right) \\
&=& \left( \begin{array}{cc}
I_{V_h} ~&~ T^{-1}Y + T^{-1} Y S^{-1} Z T^{-1}Y - T^{-1} Y S^{-1}G   \\
0 & -S^{-1} Z T^{-1}Y + S^{-1} G
\end{array} \right) \\
&=& \left( \begin{array}{cc}
I_{V_h} ~&~ T^{-1}Y - T^{-1} Y S^{-1} ( G - Z T^{-1}Y )   \\
0 & S^{-1} ( G - Z T^{-1}Y )
\end{array} \right) \\
&=& \left( \begin{array}{cc}
I_{V_h} ~&~ T^{-1}Y - T^{-1} Y S^{-1} S   \\
0 & S^{-1} S
\end{array} \right) \\
&=& \left( \begin{array}{cc}
I_{V_h} & 0   \\
0 & I_{V_{\bot}}
\end{array} \right),
\end{eqnarray*}
where $I_{V_h}: V_h \rightarrow V_h$ denotes an identity operator on $V_h$.  \\
Then, we have
\begin{eqnarray*}
D\bar{D} &=& \left( \begin{array}{cc}
T & Y \\
Z & G
\end{array} \right) \left( \begin{array}{cc}
T^{-1} + T^{-1} Y S^{-1} Z T^{-1}   ~ & ~ -T^{-1} Y S^{-1}  \\
-S^{-1} Z T^{-1} & S^{-1}
\end{array} \right) \\
&=& \left( \begin{array}{cc}
I_{V_h} & 0 \\
Z T^{-1} + Z T^{-1} Y S^{-1} Z T^{-1} - G S^{-1} Z T^{-1} ~~ &  -ZT^{-1} Y S^{-1}  + G S^{-1}
\end{array} \right) \\
&=& \left( \begin{array}{cc}
I_{V_h} & 0 \\
Z T^{-1} - ( G - Z T^{-1} Y ) S^{-1} Z T^{-1} ~~ & ~ ( G -ZT^{-1} Y ) S^{-1}
\end{array} \right) \\
&=& \left( \begin{array}{cc}
I_{V_h} & 0 \\
Z T^{-1} - S S^{-1} Z T^{-1} ~~ & ~ S S^{-1} \\
\end{array} \right) \\
&=& \left( \begin{array}{cc}
I_{V_h} & 0 \\
0 & I_{V_{\bot}}
\end{array} \right).
\end{eqnarray*}
Therefore, $\bar{D}$ is the inverse operator of $D$ and is equal to $H$. Multiplying this from the left of (\ref{eq:linear_prog_block}), we obtain
\begin{align}
	\label{eq:explicit_form}
	& \left(
	\begin{array}{c}
		\phi_h \\
		\phi_{\bot}
	\end{array}
	\right)
	= 
	\bar{D}
	\left(
	\begin{array}{c}
		{\mathcal A}_h^{-1} g \\
		{\mathcal A}_{\bot}^{-1} g
	\end{array}
	\right)
	= 
	H
	\left(
	\begin{array}{c}
		{\mathcal A}_h^{-1} g \\
		{\mathcal A}_{\bot}^{-1} g
	\end{array}
	\right).
\end{align}
Note that the above arguments provide a detailed expression of the operator matrix $H$ in the theorem.  

We next show that the linearized operator ${\mathcal L}$ is nonsingular.
To prove that ${\mathcal L}$ is injective, we show that $\phi = 0$ is the only solution to the equation ${\mathcal L} \phi = 0$. This can be readily seen from the fact that, using the nonsingularity of the operator matrix $D$ and the bijectivity of ${\mathcal A}$, the solution  for (\ref{eq:linear_prog_block}) with $g = 0$ implies that $\phi = 0$. 

Finally, we prove that the linearized operator ${\mathcal L}$ is surjective.
For this, it is sufficient to show that there exists a solution $\phi \in H^1_0(\Omega)$ satisfying ${\mathcal L} \phi = g$ for any $g \in H^{-1}(\Omega)$. Now, for $g \in H^{-1}(\Omega)$, define $( \phi_h,  \phi_{\bot})^T$ as the left-hand side of (\ref{eq:explicit_form}) and set $\phi  := \phi_h + \phi_{\bot}$.
Then, by the nonsingularity of the operator matrix $H$, the function $\phi$ satisfies  (\ref{eq:linear_prog_block}), and therefore it also  implies \eqref{eq:linear_prog}, which proves the desired surjectivity. Thus, the operator ${\mathcal L}$ is nonsingular.

\qed
\end{proof}



Theorem \ref{thm:inverse_theorem} provides a concrete expression of the operator matrix $H$ satisfying (\ref{def:H1}).
Moreover, Theorem \ref{thm:inverse_theorem} is a new expression for the solution $\phi$ of  the linear noncoercive elliptic PDE ${\mathcal L} \phi = g$.
Therefore, for example, the exact expression of the Ritz projection error for the solution of the linear noncoercive elliptic PDE ${\mathcal L} \phi = g$ is also derived from Theorem \ref{thm:inverse_theorem} (see \ref{sec:Ritz_projection_error} for details).

At the end of this section, we describe how to verify that $T$ and $S$ in the assumptions of Theorem \ref{thm:inverse_theorem} are nonsingular operators.
Let ${\vec G} \in {\mathbb R}^{N \times N}$ be a real matrix defined as $( {\vec G} )_{i, j} := ( \nabla \psi_j, \nabla \psi_i)_{L^2} - (f'[\hat{u}] \psi_j, \psi_i)_{L^2}$.
If the matrix ${\vec G}$ is invertible, then $T$ is a nonsingular operator.
Therefore, by confirming the regularity of the matrix ${\vec G}$ using numerical computation with guaranteed accuracy, we can easily check the regularity of $T$.

To confirm the regularity of the operator $S$, the following operator norm is used.
For two Banach spaces $X$ and $Y$, the set of bounded linear operators from $X$ to $Y$ is denoted by $L(X, Y)$ with the usual sup norm $\| T \|_{ L(X, Y) } := \sup\{ \| T u \|_{Y} / \| u \|_{X} : {u \in X \setminus \{0\}}\}  $ for $T \in L(X, Y)$.
When $X = Y$, we simply use $L(X)$.
Then, we can confirm the regularity of $S$ using the following well-known theorem.

\begin{lemma}[Well known]\label{Lem:Neumann}
Let $S$ be a bounded linear operator satisfying (\ref{def:S}).
Setting $\kappa := \|  {\mathcal A}_{\bot}^{-1} f'[\hat{u}] |_{V_\bot} + {\mathcal A}_{\bot}^{-1} f'[\hat{u}] |_{V_h} T^{-1} {\mathcal A}_h^{-1} f'[\hat{u}] |_{V_\bot} \|_{L(H^1_0)}$, if $\kappa < 1$ holds, then $S$ is invertible and its norm satisfies
\begin{align*}
\| S^{-1} \|_{ L(H^1_0)} \le \frac{1}{ 1 - \kappa }.
\end{align*}
\end{lemma}

\section{Numerical examples}\label{sec:numeric_example}

\subsection{Verification procedure}\label{sec:verification_procedure}
In this subsection, we describe a verification procedure to realize computer-assisted proofs using the fixed point formulation  (\ref{eq:fixed_point_form_with_H}) with the operator matrix $H$.
The proposed verification method is a combination of the FN and IN methods developed above.

Let $\hat{u} \in V_h \subset H^1_0(\Omega)$ be a solution that satisfies the following equation:
\begin{align}
\label{eq:approx_uh}
(\nabla \hat{u}, \nabla v_h)_{L^2} = (f(\hat{u}), v_h)_{L^2}, ~ \forall v_h \in V_h.
\end{align}
For example, $\hat{u}$ may be obtained by numerical computations with guaranteed accuracy for finite-dimensional nonlinear equations such as the Krawczyk method.
Note that $R_h {\mathcal A}^{-1} {\mathcal F}(\hat{u})$ in (\ref{eq:fixed_point_form_with_H}) becomes zero.

Next, we verify that the matrix ${\vec G}$ defined as $( {\vec G} )_{i, j} := ( \nabla \psi_j, \nabla \psi_i)_{L^2} - (f'[\hat{u}] \psi_j, \psi_i)_{L^2}$ and the operator $S$ are nonsingular.
As Theorem \ref{thm:inverse_theorem} implies that the linearized operator ${\mathcal L}$ defined as (\ref{def:linearize_operator}) is nonsingular, we can transform problem (\ref{main_problem}) into the fixed point equation (\ref{eq:fixed_point_form_with_H}) using the operator matrix $H$.
Note that the constant $\kappa$ in Lemma \ref{Lem:Neumann}, which is used for verifying the nonsingularity of the operator $S$, can be obtained using a method developed in previous studies \cite{nakao2005numerical, nakao2015some,kinoshita2019alternative,watanabe2019improved}.
In particular, \cite{watanabe2019improved} is a good reference regarding the differences resulting from the decomposition of the norm.

Schauder's fixed point theorem or  Banach's fixed point theorem may be applied to the fixed point equation (\ref{eq:fixed_point_form_with_H}) with the candidate sets (\ref{CandidateSet_wh}) and (\ref{CandidateSet_wbot}), as in \cite{nakao1988numerical, nakao1990numerical}.
Here, the fixed point theorem may be selected as follows.
If the nonlinear term $f(u)$ is similar to that in \cite{plum1994enclosures} (e.g., $f(u) \in L^2(\Omega)~ \forall u \in H^1_0(\Omega)$), then because ${\mathcal L}^{-1} {\mathcal G}$ is a compact operator, Schauder's fixed point theorem can be used.
If this is not the case, or if we need to prove the local uniqueness of the solution, it is preferable to use Banach's fixed point theorem.
In fact, a survey of the FN method \cite{nakao2001numerical} makes it easy to select the appropriate method.

\subsection{Example}
In this subsection, we present an example where our method is used to verify a solution of the elliptic boundary value problem
\begin{align}
\left\{
\begin{array}{ll}
-\Delta{u} = u^2 & {\rm{in}} \hspace{0.1cm} \Omega{,}\\
u = 0                                & \partial \Omega{,}
\end{array}
\right.
\label{example_program}
\end{align}
with $\Omega = (0,1)^2$.
This is Emden's equation, and is a good test problem for comparing the proposed method with other approaches because the nonlinear term $u^2$ is widespread. In addition, because the results for $\rho$ in (\ref{CandidateSet_IN}) given by many existing IN methods \cite{plum1994enclosures, plum2008, nakao2005numerical, takayasu2013verified} are almost the same, there is no need to compare numerous existing IN approaches.
Emden's equation has also been discussed in terms of the FN method \cite{watanabe1993numerical}.

All computations were implemented on a computer with 2.20 GHz Intel Xeon E7-4830 v2 CPU $\times$ 4, 2 TB
RAM, and CentOS 7.2 using C++11 with GCC version 4.8.5.
All rounding errors were strictly estimated using the toolbox kv Library \cite{kashiwagikv}.
This guarantees the mathematical correctness of all the results.

We constructed approximate solutions for \eqref{example_program} using a Legendre polynomial basis.
More concretely, define the set $\{ \psi_1, \psi_2, \cdots \psi_N \}$ of Legendre polynomials as
\begin{align*}
\psi_i(x) := \frac{1}{i(i+1)}x(1-x)\frac{dP_i}{dx}(x), ~ i = 1,2, \cdots,
\end{align*}
with
\begin{align*}
P_i = \frac{(-1)^i}{i!}\left( \frac{d}{dx} \right)^i x^i (1 - x)^i
\end{align*}
and define the finite-dimensional subspace as a tensor product 
\begin{align*}
V_h^N :=  \mbox{span}( \psi_1, \cdots \psi_N ) \times \mbox{span}( \psi_1, \cdots \psi_N ).
\end{align*}
Then, $\hat{u} \in V_h^N$ satisfying the finite-dimensional nonlinear problem (\ref{eq:approx_uh}) can be written as
\begin{align*}
\hat{u}(x, y) = \sum_{i,j = 1}^N \hat{u}_{i,j} \psi_i(x) \psi_j(y),
\end{align*}
where $\hat{u}_{i,j}$ is a real number.
For example, when $N=10$, a solution $\hat{u}$ satisfying (\ref{eq:approx_uh}) can be obtained using the Krawczyk method (see Table 1).
Here, $1.23_{456}^{789}$ denotes the interval $[1.23456, 1.23789]$.
As the solution has symmetry, note that $\hat{u}_{i,j}$ is zero when $i$ and $j$ are even.
Under this setting, approximate solutions $\hat{u} \in V_h^{10}$ were computed numerically. Their graphs are displayed in Figure \ref{Fig:approximate_splution}.
Thus, taking the candidate set as 
\begin{align}
	\label{eq:camdodate_set_Wh2d}
		W_h &:= \left\{ \sum_{i,j = 1}^{N} W_{i,j} \psi_i \subset V_h^N ~ | ~ W_i ~ \mbox{is a closed interval in } {\mathbb R} ~   \right\}, \\
	\label{eq:camdodate_set_Wbot2d}
		W_\bot &:= \{ w_{\bot} \in V_{\bot} ~ | ~ \| w_{\bot} \|_{H^1_0}, \le \alpha \},
\end{align}
the method proposed in this paper succeeded  in the numerical verification   of problem \eqref{example_program}. The verified results $W_h$ for the finite-dimensional parts are presented in Tables \ref{Tab:approximate_uh} and \ref{Tab:guarantee_result_norm}, and the results $\alpha$ for the infinite-dimensional parts are given in Table \ref{Tab:guarantee_result_norm}.
Furthermore, we can prove that the exact solution $u^*$ of (\ref{example_program}) exists in $\hat{u} + W_h + W_\bot$, and we can estimate
\begin{align*}
	\| u^* - \hat{u} \|_{H^1_0} &= \sqrt{ \| R_h (u^* - \hat{u}) \|_{H^1_0}^2 + \| (I - R_h) (u^* - \hat{u}) \|_{H^1_0}^2 } \\
		&\le \sqrt{ \sup \| W_h \|_{H^1_0}^2 + \alpha^2 } =: \rho.
\end{align*}

\begin{table}[htbp]
	\caption{For the case $N = 10$, the coefficient $\hat{u}_{i,j}$ of $\hat{u}$ and the coefficient $W_{i,j}$ of the guaranteed result $W_h$ of the finite-dimensional part}
	\label{Tab:approximate_uh}
	\begin{center}
		\renewcommand{\arraystretch}{1.2}
		\begin{tabular}{ | c | c || c || c | }
			\hline
			$i$ & $j$  & $\hat{u}_{i,j}$ & $W_{i,j}$  \\ \hline \hline
			1   &  1    & $366.4134708189518_4^5$ & $[-0.59855339300191369,0.59855339300191369]$  \\ \hline
			1   &  3    & $-152.7013688554553_4^3$ & $[-1.2253950367566178,1.2253950367566178]$\\ \hline 
			1   &  5    & $51.38270599821782_1^2$ & $[-0.98711934879483799,0.98711934879483799]$  \\ \hline 
			1   &  7    & $-10.39254971836046_7^6$ & $[-0.80157453121573663,0.80157453121573663]$ \\ \hline   
			1   &  8    & $1.938028134034345_4^5$ & $[-0.6402005288216297,0.6402005288216297]$   \\ \hline
			3   &  1    & $-152.7013688554553_4^3$ & $[-1.2253950367566338,1.2253950367566338]$ \\ \hline  
			3   &  3    & $106.206570760356_{49}^{50}$ & $[-3.4911377868360356,3.4911377868360356]$ \\ \hline  
			3   &  5    & $-47.02334109624576_2^1$ & $[-5.2643807800044531,5.2643807800044531]$ \\ \hline
			3   &  7    & $11.16273975521142_2^3$ & $[-5.4645410263719932,5.4645410263719932]$   \\ \hline
			3   &  9    & $-2.60410486533461_{90}^{89}$ & $[-4.6121337856900269,4.6121337856900269]$ \\ \hline   
			5   &  1    & $51.38270599821782_1^2$ & $[-0.98711934879489661,0.98711934879489661]$   \\ \hline
			5   &  3    & $-47.02334109624576_2^1$ & $[-5.2643807800046299,5.2643807800046299]$   \\ \hline
			5   &  5    & $23.64100645585704_8^9$ & $[-10.278811561149042,10.278811561149042]$   \\ \hline
			5   &  7    & $-6.393838247547973_9^8$ & $[-12.749567305746329,12.749567305746329]$   \\ \hline
			5   &  9    & $1.615528823767520_5^6$ & $[-10.862495129381467,10.862495129381467]$   \\ \hline
			7   &  1    & $-10.39254971836046_7^6$ & $[-0.80157453121577616,0.80157453121577616]$   \\ \hline
			7   &  3    & $11.16273975521142_2^3$ & $[-5.4645410263722019,5.4645410263722019]$   \\ \hline
			7   &  5    & $-6.393838247547973_9^8$ & $[-12.749567305746491,12.749567305746491]$   \\ \hline
			7   &  7    & $2.118831906690744_1^2$ & $[-17.471111729696386,17.471111729696386]$   \\ \hline
			7   &  9    & $-0.6015754707587541_5^4$ & $[-15.400941972506744,15.400941972506744]$   \\ \hline
			9   &  1    & $1.938028134034345_4^5$ & $[-0.64020052882164458,0.64020052882164458]$   \\ \hline
			9   &  3    & $-2.60410486533461_{90}^{89}$ & $[-4.6121337856902125,4.6121337856902125]$ \\ \hline  
			9   &  5    & $1.615528823767520_5^6$ & $[-10.862495129381687,10.862495129381687]$   \\ \hline
			9   &  7    & $-0.6015754707587541_4^4$ & $[-15.400941972506861,15.400941972506861]$   \\ \hline
			9   &  9    & $0.1963785013053931_7^8$ & $[-13.81780174057843,13.81780174057843]$ \\ \hline
		\end{tabular}
	\end{center}
\end{table}

\begin{table}[htbp]
	\caption{Result of norm evaluation for $N = 10$}
	\label{Tab:guarantee_result_norm}
	\begin{center}
		\renewcommand{\arraystretch}{1.2}
		\begin{tabular}{| c || c | c || c | }
			\hline
			method & $\sup \| W_h \|_{H^1_0} \le$ & $\sup \| W_\bot \|_{H^1_0} \le \alpha$  & $\| u^* - \hat{u} \|_{H^1_0} \le \rho$ \\ \hline \hline
			Proposed & $0.41226282803760456$   &  $0.14598888170328537$    & $0.43734813702877418$ \\ \hline
			Existing IN		 & --   & --    & $1.4392104268509974$ \\ \hline
		\end{tabular}
	\end{center}
\end{table}

FN-Int (e.g., \cite{nakao1990numerical, nakao2001numerical, nakao2011numerical}) was also applied using similar candidate sets (\ref{eq:camdodate_set_Wh2d}) and (\ref{eq:camdodate_set_Wbot2d}) and performing a detailed evaluation of the finite-dimensional and infinite-dimensional parts.
However, as FN-Int only applies Newton's method to the finite-dimensional parts, the verification fails for $N=10$, which implies $N=10$ is too small to achieve successful verification.

As the error bound of the form  $\| u^* - \hat{u} \|_{H^1_0} \le \rho$ is obtained in the course of successful verification by the IN method for $N=10$,  we compare the proposed method with the IN method  from this viewpoint in Table \ref{Tab:guarantee_result_norm}. It is clear from the table that the value of $\rho$ given by the proposed method is smaller than that produced by the existing IN method.
Additionally, as described in subsection \ref{sec:verification_procedure}, because we partially incorporate the detailed calculations of the IN method  into the proposed method, enhancing the IN method could lead to improvements in the results of the proposed method.
The reason why the proposed method produces smaller values of $\rho$ than the existing IN method is discussed in \ref{sec:relation_invnorm_Theorem} and \ref{sec:NK-theorem}.

\begin{figure}[h]
		\begin{center}
		\includegraphics[width=80mm]{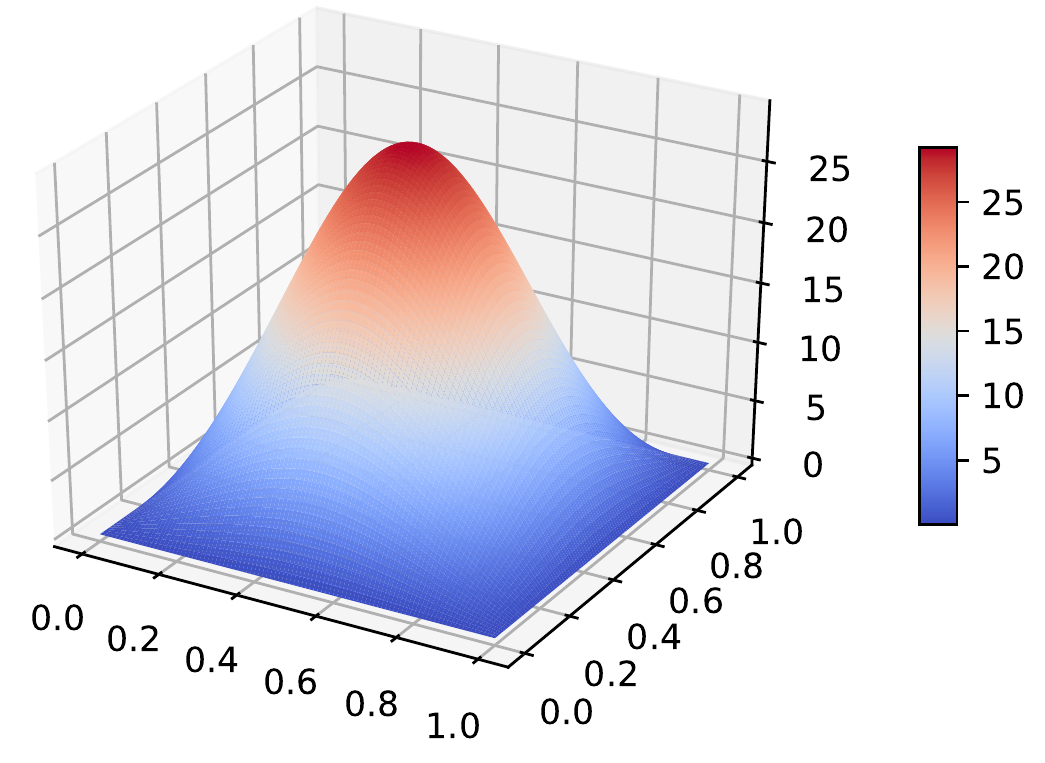}	
		\end{center}
		\caption{An approximate solution of \eqref{example_program} ($N=10$)}
	\label{Fig:approximate_splution}
\end{figure}

From Table \ref{Tab:approximate_uh}, which presents results for $N = 10$, the error interval of the finite-dimensional part appears to be rather large.
This is because $N$ is too small. In fact, very sharp results can be obtained in the case of $N = 40$ (see Table \ref{Tab:approximate_uh_40}).
Here, $\pm 1.23_{456}^{789} \times 10^{-11}$ denotes the interval $[-1.23456 \times 10^{-11}, 1.23789 \times 10^{-11}]$.

\begin{table}[htbp]
	\caption{For the case $N = 40$, the coefficient $\hat{u}_{i,j}$ of $\hat{u}$ and the coefficient $W_{i,j}$ of the guaranteed result $W_h$ of the finite-dimensional part}
	\label{Tab:approximate_uh_40}
	\begin{center}
		\renewcommand{\arraystretch}{1.2}
		\begin{tabular}{ | c | c || c || c | }
			\hline
			$i$ & $j$  & $\hat{u}_{i,j}$ & $W_{i,j}$  \\ \hline \hline
1 & 1 & $ 366.4132119391286_8^9 $ & $ \pm 2.890339699_{1423603}^{0254412} \times 10^{-11} $ \\ \hline
1 & 3 & $ -152.7015836846106_1^0$ & $ \pm5.953039492_{5126237}^{6304853} \times 10^{-11} $ \\ \hline
1 & 5 & $ 51.38361560840872_6^7 $ & $ \pm 4.9329670131_{634702}^{167493} \times 10^{-11} $ \\ \hline
1 & 7 & $ -10.39601151027711_1^0 $ &$ \pm 4.3133385563_{658583}^{703399} \times 10^{-11} $ \\ \hline
1 & 9 & $ 1.946100573383924_1^2 $ & $ \pm 4.05963569250_{10538}^{54126} \times 10^{-11} $ \\ \hline
$\cdots$ & $\cdots$ &  $\cdots$  &  $\cdots$  \\ \hline
1 & 39 & $-4.66924461_{568399560}^{123737562} \times 10^{-12} $ & $ \pm 2.40673056579_{81076}^{49860} \times 10^{-11} $ \\ \hline
3 & 1 & $ -152.7015836846106_1^0 $ & $ \pm 5.953039492_{5723655}^{6902272} \times 10^{-11} $ \\ \hline
3 & 3 & $ 106.2071425465360_7^8  $ & $ \pm 1.7449437714_{236332}^{139130} \times 10^{-10} $ \\ \hline
3 & 5 & $ -47.02462830095045_7^6 $ & $ \pm 2.780252688_{5991257}^{6047521} \times 10^{-10} $ \\ \hline
3 & 7 & $ 11.16900188749037_0^1 $ & $ \pm 3.25583866986_{64090}^{45918} \times 10^{-10} $ \\ \hline
3 & 9 & $ -2.604322866563870_7^6 $ & $\pm 3.59972699307_{48283}^{54461} \times 10^{-10} $ \\ \hline
$\cdots$ & $\cdots$ &  $\cdots$  &  $\cdots$  \\ \hline
3 & 39 & $ -6.18701438_{85354111}^{75591113} \times 10^{-11} $ & $ \pm 2.432811717189_{8549}^{6228} \times 10^{-10} $ \\ \hline
5 & 1 & $ 51.38361560840872_6^7 $ & $ \pm 4.932967013_{3060612}^{2593404} \times 10^{-11} $ \\ \hline
5 & 3 & $ -47.02462830095045_7^6 $ & $ \pm 2.7802526886_{393443}^{449707} \times 10^{-10} $ \\ \hline
5 & 5 & $ 23.6420677490914_{49}^{50} $ & $ \pm 5.90180296672_{40844}^{13061} \times 10^{-10} $ \\ \hline
5 & 7 & $ -6.399132323486540_8^7 $ & $ \pm 8.5522929274_{298881}^{306522} \times 10^{-10} $ \\ \hline
5 & 9 & $ 1.610834571271026_0^1 $ & $ \pm 1.03788894888648_{46}^{85} \times 10^{-09} $ \\ \hline
$\cdots$ & $\cdots$ &  $\cdots$  &  $\cdots$  \\ \hline
5 & 39 & $ -2.4529205624_{387362}^{231429} \times 10^{-10} $ & $ \pm 7.790762135816_{5518}^{2716} \times 10^{-10} $ \\ \hline
7 & 1 & $ -10.39601151027711_1^0 $ & $ \pm 4.3133385564_{792260}^{837076} \times 10^{-11} $ \\ \hline
7 & 3 & $ 11.16900188749037_0^1 $ & $ \pm 3.25583866991_{72636}^{54463} \times 10^{-10} $ \\ \hline
7 & 5 & $ -6.399132323486540_8^7 $ & $ \pm 8.55229292747_{66569}^{74210} \times 10^{-10} $ \\ \hline
7 & 7 & $ 2.124267226042428_2^3 $ & $ \pm 1.43656500261616_{32}^{22} \times 10^{-09} $ \\ \hline
7 & 9 & $ -0.5986559904211578_9^8 $ & $ \pm 1.936424374565_{5089}^{4920} \times 10^{-09} $ \\ \hline
$\cdots$ & $\cdots$ &  $\cdots$  &  $\cdots$  \\ \hline
7 & 39 & $ -6.160592931_{6028183}^{4758797} \times 10^{-10} $ & $ \pm 1.6509622637659_{917}^{898} \times 10^{-09} $ \\ \hline
9 & 1 & $ 1.946100573383924_1^2 $ & $ \pm 4.0596356925_{874838}^{918427} \times 10^{-11} $ \\ \hline
9 & 3 & $ -2.604322866563870_7^6 $ & $ \pm 3.59972699312_{58183}^{64361} \times 10^{-10} $ \\ \hline
9 & 5 & $ 1.610834571271026_0^1 $ & $ \pm 1.03788894889364_{30}^{69} \times 10^{-09} $ \\ \hline
9 & 7 & $ -0.5986559904211578_9^8 $ & $ \pm 1.9364243745704_{203}^{034} \times 10^{-09} $ \\ \hline
9 & 9 & $ 0.1907886733421144_7^8 $ & $ \pm 2.8485078315651_{604}^{930} \times 10^{-09} $ \\ \hline
$\cdots$ & $\cdots$ &  $\cdots$  &  $\cdots$  \\ \hline
9 & 39 & $ -1.2284714549_{968888}^{703340} \times 10^{-09} $ & $ \pm 2.8184100199765_{508}^{181} \times 10^{-09} $ \\ \hline
$\cdots$ & $\cdots$ &  $\cdots$  &  $\cdots$  \\ \hline
39 & 39 & $ 9.32486073_{24817149}^{46069927} \times 10^{-10} $ & $ \pm 5.810714309286_{8021}^{9468} \times 10^{-09} $ \\ \hline
		\end{tabular}
	\end{center}
\end{table}

\section{Conclusion}
In this paper, we have described a new formulation (\ref{eq:fixed_point_form_with_H}) for the numerical proof of the existence of solutions for elliptic problems. The proposed approach has advantages over both the existing FN and IN methods.   
In particular, we derived a specific formula for the operator matrix $H$, which is needed to compute (\ref{eq:fixed_point_form_with_H}), in Theorem \ref{thm:inverse_theorem}.
As a result, while using the infinite-dimensional Newton's method, the error evaluation for each coefficient interval of the finite basis is also enclosed, as demonstrated by the results in Table \ref{Tab:approximate_uh}.  This is considered an advantage of the FN method.
Furthermore, the proposed method produces better results than the IN method, even for the norm estimation, as demonstrated by Table \ref{Tab:guarantee_result_norm}.

\renewcommand{\appendix}{}
\appendix
\setcounter{section}{0}
\renewcommand{\thesection}{Appendix \Alph{section}}
\renewcommand{\thesubsection}{Appendix \Alph{section}-\arabic{subsection}.}

\section{Another formula for the operator matrix $H$}
In Theorem \ref{thm:inverse_theorem}, the Schur complement $S$ (defined in (\ref{def:S})) for the $(1, 1)$ element $T$ of the operator matrix $D$ is created and some properties are proved. 
We present another form of the operator matrix $H$ using the Schur complement $S_h: V_h \rightarrow V_h$ for the $(2, 2)$ element $I_{V_{\bot}} - {\mathcal A}_{\bot}^{-1} f'[\hat{u}] |_{V_\bot} (=: G)$ of the operator matrix $D$.

\begin{theorem}\label{thm:inverse_theorem2}
The infinite-dimensional operator $G: V_{\bot} \rightarrow V_{\bot}$ defined in (\ref{def:YZG}) is assumed to be nonsingular.
Let $S_h: V_h \rightarrow V_h$ be a linear operator defined as
\begin{align}
	\label{def:Sh}
		S_h := T - {\mathcal A}_h^{-1} f'[\hat{u}] |_{V_\bot} G^{-1} {\mathcal A}_{\bot}^{-1} f'[\hat{u}] |_{V_h}.
\end{align}
If $S_h$ is nonsingular, then the operators $D$ and  ${\mathcal L}$ defined by (\ref{def:operator_matrix_D}) and (\ref{def:linearize_operator}), respectively, are also nonsingular, and the solution $(\phi_h, \phi_{\bot})^T \in V_h \times V_{\bot}$ that satisfies (\ref{eq:linear_prog_block}) is represented as 
{\small
\begin{align*}
	& \left(
	\begin{array}{c}
		\phi_h \\
		\phi_{\bot}
	\end{array}
	\right) \\ 
	= & \left(
	\begin{array}{c c}
		S_h^{-1} & S_h^{-1} {\mathcal A}_h^{-1} f'[\hat{u}] |_{V_\bot} G^{-1}  \\
		G^{-1} {\mathcal A}_{\bot}^{-1} f'[\hat{u}] |_{V_h} S_h^{-1} ~~ & ~~ G^{-1} + G^{-1} {\mathcal A}_{\bot}^{-1} f'[\hat{u}] |_{V_h} S_h^{-1} {\mathcal A}_h^{-1} f'[\hat{u}] |_{V_\bot} G^{-1}
	\end{array}
	\right)
	\left(
	\begin{array}{c}
		{\mathcal A}_h^{-1} g \\
		{\mathcal A}_{\bot}^{-1} g
	\end{array}
	\right).
\end{align*}
}
\end{theorem}

\begin{proof}
Using the notation in (\ref{def:YZG}), we define the operator matrix $\bar{D}$ as 
\begin{align*}
\bar{D} := \left( \begin{array}{cc}
S_h^{-1} & -S_h^{-1} Y G^{-1} \\
-G^{-1} Z S_h^{-1} ~~ & ~ G^{-1} + G^{-1} Z S_h^{-1} Y G^{-1} 
\end{array} \right).
\end{align*}
Then, $\bar{D}$ coincides with the inverse operator matrix of  $D$ in \eqref{def:operator_matrix_D} in the same way as in the proof of Theorem \ref{thm:inverse_theorem}.

\qed
\end{proof}

As we can verify the nonsingularity of the operator $G$ in a similar manner to that in Lemma \ref{Lem:Neumann}, it is also possible to derive a verification procedure using Theorem \ref{thm:inverse_theorem2} instead of Theorem \ref{thm:inverse_theorem}.

Moreover, it is possible to derive, from the operator matrix $D$ in  Theorem \ref{thm:inverse_theorem}, the results in previous papers \cite{nakao2005numerical, nakao2015some,watanabe2019improved} for the norm evaluation $\| {\mathcal L}^{-1} \|_{H^{-1}, H^1_0}$ of the inverse operator ${\mathcal L}^{-1}$ (see \ref{sec:relation_invnorm_Theorem}).
However, there has been no previous discussion of the norm evaluation of the operator $ {\mathcal L}^{-1} \; : \; H^{-1} \rightarrow  H^1_0$ using the operator matrix $D$ in Theorem \ref{thm:inverse_theorem2}.

\section{Relation between the norm $\| {\mathcal L}^{-1} \|_{L(H^{-1}, H^1_0)}$ and Theorem \ref{thm:inverse_theorem}} \label{sec:relation_invnorm_Theorem}
There have been many studies on the estimation of the norm $\|  {\mathcal L}^{-1}  \|_{L(H^{-1},  H^1_0)}$, because it plays an important role in the existing IN method (e.g., \cite{plum1991bounds, plum1994enclosures, oishi1995numerical, nakao2005numerical, tanaka2014verified, nakao2015some, watanabe2019improved}). In the following, we show that it is also possible to obtain the upper bound of the inverse operator norm $\| {\mathcal L}^{-1} \|_{L(H^{-1}, H^1_0)}$ using Theorem \ref{thm:inverse_theorem}.

\begin{corollary}[of Theorem \ref{thm:inverse_theorem}]\label{cor:inverse_operator}
Under the same assumptions as in Theorem \ref{thm:inverse_theorem}, ${\mathcal L}$ is invertible and it follows that
{\small
\begin{align*}
& \| {\mathcal L}^{-1} \|_{L(H^{-1}, H^1_0)} \le \\
& \left\| \! \left( \! \begin{array}{c c}
		\| T^{-1}  \! +  \! T^{-1}{\mathcal A}^{-1}_h f'[\hat{u}]|_{V_{\bot}} S^{-1} {\mathcal A}^{-1}_{\bot} f'[\hat{u}]|_{V_h}T^{-1}\|_{L(H^1_0)} \! & \! \| T^{-1}{\mathcal A}^{-1}_h f'[\hat{u}]|_{V_{\bot}}S^{-1} \|_{L(H^1_0)} \\
		\| S^{-1}  {\mathcal A}^{-1}_{\bot} f'[\hat{u}]|_{V_h}T^{-1} \|_{L(H^1_0)}   & \| S^{-1} \|_{L(H^1_0)}
	\end{array} \! \right) \! \right\|_{E} \! ,
\end{align*} }
where $\| \cdot \|_{E}$ denotes the Euclidean norm.
\end{corollary}

\begin{proof}

We define the norm of the direct product space $V_h \times V_{\bot}$ as
{\small
\begin{align*}
\left\| \left( \begin{array}{c}
\phi_h \\
\phi_\bot
\end{array} \right) \right\|_{V_h \times V_{\bot}} := \left\| \left( \begin{array}{c}
\| \phi_h \|_{H^1_0} \\
\| \phi_\bot \|_{H^1_0}
\end{array} \right) \right\|_{E} = \sqrt{ \| \phi_h \|_{H^1_0}^2 + \| \phi_\bot \|_{H^1_0}^2}, ~ (\phi_h, \phi_\bot )^T \in V_h \times V_{\bot}.
\end{align*}}
Then, from Theorem \ref{thm:inverse_theorem}, the solution $\phi$ of the linear equation (\ref{eq:linear_prog}) can be  evaluated as
\begin{align*}
\| \phi \|_{H^1_0} =& \left\| \left( \begin{array}{c}
\| R_h \phi \|_{H^1_0} \\
\| (I-R_h) \phi \|_{H^1_0}
\end{array} \right) \right\|_{E} 
=  \left\| H \left( \begin{array}{c}
{\mathcal A}^{-1}_h g \\
{\mathcal A}^{-1}_\bot g
\end{array} \right) \right\|_{V_h \times V_{\bot}} \\
\le& \| H \|_{L( V_h \times V_\bot )} \left\| \left( \begin{array}{c}
{\mathcal A}^{-1}_h g \\
{\mathcal A}^{-1}_\bot g
\end{array} \right) \right\|_{V_h \times V_\bot} 
= \| H \|_{L( V_h \times V_\bot )} \left\| {\mathcal L} \phi \right\|_{H^{-1}}.
\end{align*}
Therefore, we have
\begin{align*}
\| {\mathcal L}^{-1} \|_{L(H^{-1}, H^1_0)} \le \| H \|_{L( V_h \times V_\bot )}.
\end{align*}
Moreover, using the structure of the operator matrix $H$ defined in (\ref{def:H1}), we have
\begin{align*}
& \| H \|_{L( V_h \times V_\bot )} = \sup_{z = (z_h, z_\bot)^T \in V_h \times V_{\bot} } \frac{ \left\| \left( \begin{array}{c c}
H_{11} & H_{12} \\
H_{21} & H_{22}
\end{array} \right) \left( \begin{array}{c}
z_h \\
z_{\bot}
\end{array} \right) \right\|_{V_h \times V_\bot } }{ \| z \|_{V_h \times V_{\bot}} } \\
=& \sup_{z = (z_h, z_\bot)^T \in V_h \times V_{\bot} } \frac{ \left\| \left( \begin{array}{c}
\| H_{11} z_h + H_{12} z_{\bot} \|_{H^1_0} \\
\| H_{21} z_h + H_{22} z_{\bot} \|_{H^1_0}
\end{array} \right) \right\|_{E} }{ \| z \|_{V_h \times V_{\bot}} } \\
\le & \left\| \left( \begin{array}{c c}
\| H_{11} \|_{L(H^1_0)} & \| H_{12} \|_{L(H^1_0)} \\
\| H_{21} \|_{L(H^1_0)} & \| H_{22} \|_{L(H^1_0)}
\end{array} \right) \right\|_{E}.
\end{align*}
\qed
\end{proof}

\begin{remark} \label{remark:relation_Norm_Theorem}
The estimation of the inverse operator norm in Corollary~\ref{cor:inverse_operator} is closely related to the method used in previous studies \cite{nakao2005numerical,  nakao2015some, watanabe2019improved}.
However, it is recommended that Theorem \ref{thm:inverse_theorem} be applied directly without using Corollary \ref{cor:inverse_operator}.
For example, for the first term on the right-hand side of (\ref{eq:IN_Method}), the existing IN method evaluates 
\begin{align*}
\| {\mathcal L}^{-1} {\mathcal F}(\hat{u}) \|_{H^1_0} \le \| {\mathcal L}^{-1} \|_{ L(H^{-1}, H^1_0) } \| {\mathcal F}(\hat{u}) \|_{H^{-1}}
\end{align*}
to apply Corollary \ref{cor:inverse_operator}.
If $\hat{u}$ is a solution satisfying (\ref{eq:approx_uh}), then $R_h {\mathcal A}^{-1} {\mathcal F}(\hat{u})$ vanishes, but this is not reflected in the estimation of the norm $\| {\mathcal L}^{-1} \|_{ L(H^{-1}, H^1_0) }$.
In contrast, by using Theorem \ref{thm:inverse_theorem} directly, we can estimate 
\begin{align*}
\left( \begin{array}{c c}
H_{11} & H_{12} \\
H_{21} & H_{22}
\end{array} \right) \left( \begin{array}{c}
R_h {\mathcal A}^{-1} {\mathcal F}(\hat{u}) \\
(I - R_h)  {\mathcal A}^{-1} {\mathcal F}(\hat{u})
\end{array} \right) = \left( \begin{array}{c c}
0 & H_{12} \\
0 & H_{22}
\end{array} \right) \left( \begin{array}{c}
0 \\
(I - R_h)  {\mathcal A}^{-1} {\mathcal F}(\hat{u})
\end{array} \right).
\end{align*}
Therefore, the estimations of $\| H_{11} \|_{L(H^1_0)}$ and $\| H_{21} \|_{L(H^1_0)}$ are useless in the evaluation for $\| {\mathcal L}^{-1} \|_{L(H^{-1}, H^1_0)}$. Namely, the proposed method is better than existing IN methods for evaluating $\| {\mathcal L}^{-1} \|_{L(H^{-1}, H^1_0)}$.
\end{remark}

The inverse operator norm $\| {\mathcal L}^{-1} \|_{L(H^{-1}, H^1_0)}$ can be evaluated from Theorem \ref{thm:inverse_theorem2} using the same procedure as Corollary \ref{cor:inverse_operator}.

\begin{corollary}[of Theorem \ref{thm:inverse_theorem2}]\label{cor:inverse_operator2}
Under the same assumptions as in Theorem \ref{thm:inverse_theorem2}, ${\mathcal L}$ is invertible and
{\small
\begin{align*}
& \| {\mathcal L}^{-1} \|_{L(H^{-1}, H^1_0)} \le \\
& \left\| \! \left( \! \begin{array}{c c}
		\| S_h^{-1} \|_{L(H^1_0)} & \| S_h^{-1} {\mathcal A}_h^{-1} f'[\hat{u}] |_{V_\bot} G^{-1} \|_{L(H^1_0)} \\
		\| G^{-1} {\mathcal A}_{\bot}^{-1} f'[\hat{u}] |_{V_h} S_h^{-1} \|_{L(H^1_0)}  \! &  \! \| G^{-1} + G^{-1} {\mathcal A}_{\bot}^{-1} f'[\hat{u}] |_{V_h} S_h^{-1} {\mathcal A}_h^{-1} f'[\hat{u}] |_{V_\bot} G^{-1} \|_{L(H^1_0)}
	\end{array} \! \right) \! \right\|_{E}.
\end{align*} }
\end{corollary}

\section{An exact expression for the Ritz projection error in the solution of the noncoercive  equation (\ref{eq:linear_prog})}\label{sec:Ritz_projection_error}
In this section, we derive an expression for the Ritz projection error $\phi - R_h \phi$ in the exact solution $\phi$ of the linear noncoercive elliptic PDE (\ref{eq:linear_prog}).
Using this result, we can, for example, determine a constant $C$ that satisfies the inequality $\| \phi - R_h \phi \|_{H^1_0} \le C \| {\mathcal L} \phi \|_{L^2}, ~ \phi \in \{ \phi \in H^1_0 | \Delta \phi \in L^2(\Omega) \}$, even though the elliptic operator $\mathcal L$ is noncoercive (cf.  \cite{nakao2008guaranteed}).

\begin{corollary}[of Theorem \ref{thm:inverse_theorem}]\label{cor:projection_error}
Under the same assumptions as in Theorem \ref{thm:inverse_theorem}, the Ritz projection error of the solution $\phi \in H^1_0(\Omega)$ of the linear noncoercive elliptic PDE (\ref{eq:linear_prog}) is
\begin{align*}
\phi - R_h \phi = S^{-1} (I - R_h) {\mathcal A}^{-1} \left( I + f'[\hat{u}]|_{V_h}T^{-1} R_h {\mathcal A}^{-1}\right) {\mathcal L} \phi.
\end{align*}
\end{corollary}

\begin{proof}
The proof is obtained immediately from $\phi_\bot$ in Theorem \ref{thm:inverse_theorem}.

\qed
\end{proof}

Using Corollary \ref{cor:projection_error} as $g \in L^2(\Omega)$, it is easy to derive the constant $C$ satisfying $\| \phi - R_h \phi \|_{H^1_0} \le C \| {\mathcal L} \phi \|_{L^2}$ in \cite{nakao2008guaranteed}.

Similarly, we can derive the following proposition from Theorem \ref{thm:inverse_theorem2}.

\begin{corollary}[of Theorem \ref{thm:inverse_theorem2}]\label{cor:projection_error2}
Under the same assumptions as in Theorem \ref{thm:inverse_theorem2}, the Ritz projection error in the solution $\phi \in H^1_0(\Omega)$ of the linear noncoercive elliptic PDE (\ref{eq:linear_prog}) is
represented as follows:
{\small
\begin{align*}
& \phi - R_h \phi  \\
=& G^{-1} (I - R_h) {\mathcal A}^{-1} \left( I + f'[\hat{u}] |_{V_h} S_h^{-1} R_h {\mathcal A}^{-1} \left( I + f'[\hat{u}] |_{V_\bot} G^{-1} (I-R_h) {\mathcal A}^{-1} \right) \right) {\mathcal L}\phi.
\end{align*}}
\end{corollary}

\section{Verification method using a Kantorovich-type theorem based on Theorem \ref{thm:inverse_theorem}} \label{sec:NK-theorem}
In section \ref{sec:numeric_example}, we presented some actual examples of candidate sets (\ref{CandidateSet_wh}) and (\ref{CandidateSet_wbot}).
However, it has also been pointed out that the finite-dimensional candidate set (\ref{CandidateSet_wh}) has the drawback whereby the interval width increases because it involves interval arithmetic \cite{nakao2004efficient}.
This shortcoming is almost solved by our present method, because  the expression ${\mathcal G}$ defined by (\ref{def:Lipsitz}) only contains the term $(w_h + w_\bot)$ in quadratic or higher-order expressions.
However,  the proposed method has some disadvantages from the viewpoint of computational cost.
Therefore, we first introduce a Kantorovich-type theorem as an example in which the candidate set is (\ref{CandidateSet_IN}), and demonstrate the usefulness of applying Theorem \ref{thm:inverse_theorem} to this Kantorovich-type theorem.
Let $B(u, r) := \{ v \in H^1_0(\Omega) ~ | ~ \| u - v \|_{H^1_0} < r \}$ be an open ball and $\bar{B}(u, r) := \{ v \in H^1_0(\Omega) ~ | ~ \| u - v \|_{H^1_0} \le r \}$ be a closed ball.

\begin{theorem}[Kantorovich-type]\label{thm:Kantorovich}
Let $\hat{u} \in V_h \subset H^1_0(\Omega)$ be an approximate solution of (\ref{main_problem}). 
Assume that the Fr\'echet derivative ${\mathcal L}$ is nonsingular and satisfies
\begin{align}
\label{const:Kantoro_alpha}
\| {\mathcal L}^{-1} {\mathcal F}(\hat{u}) \|_{H^1_0} \le \beta
\end{align}
for a certain positive $\beta$,
and that the following holds for a certain positive $\omega$:
\begin{align}
\label{const:Kantoro_omega}
\| {\mathcal L}^{-1} \left( f'[\hat{u}] - f'[\hat{u} + v] \right) \|_{L(H^1_0)} \le \omega \| v \|_{H^1_0}, ~ v \in \bar{B}(0, 2\beta).
\end{align}
If $\beta \omega < 1/2$, then there exists a solution $u \in \bar{B}(\hat{u}, \rho)$ of (\ref{main_problem}) satisfying
\begin{align*}
\| u - \hat{u} \|_{H^1_0} \le \rho := \frac{ 1 - \sqrt{ 1 - 2 \beta \omega } }{ \omega }.
\end{align*}
Moreover, the solution is unique in $\bar{B}(0, 2\beta)$.
\end{theorem}

When we use Kantorovich-type theorems in computer-assisted existence proofs, it is necessary to determine the constants $\beta$ and $\omega$.
In existing methods (e.g., \cite{plum1994enclosures, plum2009computer, takayasu2013verified}), these constants are determined as follows:
\begin{align*}
\| {\mathcal L}^{-1} {\mathcal F}(\hat{u}) \|_{H^1_0} \le \| {\mathcal L}^{-1} \|_{L(H^{-1}, H^1_0)} \| {\mathcal F}(\hat{u}) \|_{H^{-1}} \le K \delta =: \beta,
\end{align*}
where $\| {\mathcal L}^{-1} \|_{L(H^{-1}, H^1_0)} \| \leq K, \| {\mathcal F}(\hat{u}) \|_{H^{-1}} \leq \delta$,\\
and
\begin{align*}
\| {\mathcal L}^{-1} \left( f'[\hat{u}] - f'[\hat{u} + v] \right) \|_{L(H^1_0)} \le & \| {\mathcal L}^{-1} \|_{L(H^{-1}, H^1_0)} \|  f'[\hat{u}] - f'[\hat{u} + v] \|_{L(H^1_0, H^{-1})} \\
\le & K G \| v \|_{H^1_0}, ~ v \in \bar{B}(0, 2\beta),
\end{align*}
where we have assumed that $\|  f'[\hat{u}] - f'[\hat{u} + v] \|_{L(H^1_0, H^{-1})} \leq G\| v \|_{H^1_0}$, 
respectively.
As noted in Remark \ref{remark:relation_Norm_Theorem}, the estimation using Theorem \ref{thm:inverse_theorem} gives better results than that via $\| {\mathcal L}^{-1} \|_{L(H^{-1}, H^1_0)}$.

\begin{corollary}[of Theorem \ref{thm:inverse_theorem}] \label{cor:NK_alpha}
Under the same assumptions as in Theorems \ref{thm:inverse_theorem} and \ref{thm:Kantorovich},
if $\hat{u}$ is a solution that satisfies (\ref{eq:approx_uh}), we can estimate
\begin{align*}
\| {\mathcal L}^{-1} {\mathcal F}(\hat{u}) \|_{H^1_0} = \left\| \left( \begin{array}{c}
\| T^{-1} R_h{\mathcal A}^{-1} f'[\hat{u}]|_{V_{\bot}}S^{-1} ( I - R_h) {\mathcal A}^{-1} {\mathcal F}(\hat{u}) \|_{H^1_0} \\
\| S^{-1} (I - R_h) {\mathcal A}^{-1} {\mathcal F}(\hat{u}) \|_{H^1_0}
\end{array} \right) \right\|_{E}.
\end{align*}
\end{corollary}

\begin{corollary}[of Theorem \ref{thm:inverse_theorem}] \label{cor:NK_omega}
Under the same assumptions as in Theorems \ref{thm:inverse_theorem} and \ref{thm:Kantorovich}, we can estimate
\begin{align*}
& \| {\mathcal L}^{-1} \left( f'[\hat{u}] - f'[\hat{u} + v] \right) \|_{L(H^1_0)} \\
= & \sup_{\phi \in H^1_0(\Omega)} \frac{ \left\| H \left( \begin{array}{c}
R_h {\mathcal A}^{-1}( f'[\hat{u}] - f'[\hat{u} + v] ) \phi \\
(I - R_h) {\mathcal A}^{-1}( f'[\hat{u}] - f'[\hat{u} + v] ) \phi
\end{array} \right) \right\|_{V_h \times V_\bot } }{ \| \phi \|_{H^1_0} },
\end{align*}
where $v \in \bar{B}(0, 2\beta)$.
\end{corollary}

Corollary \ref{cor:NK_alpha} clearly gives the desired estimation with higher accuracy than the existing method.
However, because Kantorovich-type theorems use the candidate set (\ref{CandidateSet_IN}) instead of (\ref{CandidateSet_wh}) and (\ref{CandidateSet_wbot}),  Corollary \ref{cor:NK_omega} is not necessarily advantageous compared with the existing methods.
For example, in the case of $f(u) = u^3$, $\hat{u}$ may be superior to the existing method because $f'[\hat{u}] - f'[\hat{u} + v] = -6\hat{u}v - 3v^2$.


\section*{Acknowledgements}
This work is supported by JST CREST Grant Number JPMJCR14D4, and by MEXT under the ``Exploratory Issue on Post-K computer'' project (Development of verified numerical computations and super high-performance computing environment for extreme research).
The second author is supported by JSPS KAKENHI Grant Number 18K03434.

\bibliographystyle{amsplain} 
\bibliography{ref.bib}

\end{document}